\documentclass[10pt, reqno]{amsart}
\usepackage{amsmath, amsthm, amscd, amsfonts, amssymb, graphicx, color}
\usepackage[bookmarksnumbered, colorlinks, plainpages]{hyperref}
\hypersetup{colorlinks=true,linkcolor=red, anchorcolor=green, citecolor=cyan, urlcolor=red, filecolor=magenta, pdftoolbar=true}
\usepackage{mathrsfs}

\newtheorem{theorem}{Theorem}[section]
\newtheorem{lemma}[theorem]{Lemma}

\newtheorem{corollary}[theorem]{Corollary}
\theoremstyle{definition}

\theoremstyle{remark}
\newtheorem{remark}[theorem]{Remark}
\numberwithin{equation}{section}

\begin{document}
\setcounter{page}{1}

\title[Untying knots in 4D]{Untying knots in 4D and Wedderburn's theorem}

\author[Nikolaev]
{Igor ~V. ~Nikolaev$^1$}

\address{$^{1}$ Department of Mathematics and Computer Science, St.~John's University, 8000 Utopia Parkway,  
New York,  NY 11439, United States.}
\email{\textcolor[rgb]{0.00,0.00,0.84}{igor.v.nikolaev@gmail.com}
}





\begin{abstract}
It is proved that the Wedderburn Theorem on  finite division rings implies 
that all knots and links in  the smooth 4-dimensional manifolds are  trivial. 
\end{abstract}

\maketitle

\section{Introduction}
Our brief note contains an algebraic proof of the otherwise known topological fact,  that 
all knots and links in the smooth 4-dimensional manifolds can be untied, i.e. are trivial. 
The novelty is a surprising  r\^ole of the Wedderburn Theorem  [Maclagan-Wedderburn 1905] \cite{Wed1}
in the 4-dimensional topology \cite{Nik1}. 

\medskip
Recall that arithmetic topology studies  a functor, $F$,   between the 3-dimensional manifolds
 and the fields of algebraic numbers [Morishita 2012] \cite{M}.  
Such a functor maps   $3$-dimensional manifolds $\mathscr{M}^3$ to 
 the algebraic  number fields $K$, so that the knots  (links, resp.)  in   $\mathscr{M}^3$ correspond  to 
the prime ideals (ideals, resp.)  in the ring of integers $O_K$.

The map  $F$ extends to the smooth 4-dimensional manifolds $\mathscr{M}^4$ and the fields of
hyper-algebraic numbers $\mathbb{K}$, i.e. fields with a non-commutative multiplication \cite{Nik1}. 
To formulate our result, denote by $O_{\mathbb{K}}$ the ring of integers of the field $\mathbb{K}$. 
A ring $R$ is called a  domain, if $R$ has no zero divisors.  The  $R$ is called simple,
if it has only trivial two-sided ideals.  Our main result is the following theorem. 
\begin{theorem}\label{thm1.1}
$O_{\mathbb{K}}$ is a simple domain. 
\end{theorem}
\begin{remark}\label{rmk1.2}
Theorem \ref{thm1.1} is false for the  algebraic integers, since the domain $O_K$
is never simple. 
\end{remark}
\begin{corollary}\label{cor1.3}
Any knot or link in $\mathscr{M}^4$ is trivial. 
\end{corollary}
\begin{proof}
If  $\mathscr{K}\subset \mathscr{M}^4$  ($\mathscr{L}\subset \mathscr{M}^4$, resp.)  
is a non-trivial knot (link, resp.),  then $F(\mathscr{K})$ ($F(\mathscr{L})$, resp.) is a 
non-trivial two-sided prime ideal (two-sided ideal, resp.) in  $O_{\mathbb{K}}$. 
The latter contradicts \ref{thm1.1},  since  $O_{\mathbb{K}}$ is a simple ring. 
\end{proof}

\bigskip
The paper is organized as follows. Section 2 contains a brief review 
of the preliminary results. Theorem \ref{thm1.1} is proved in Section 3.

\section{Preliminaries}
\subsection{Arithmetic topology}
The arithmetic topology studies an interplay between
3-dimensional manifolds and  number fields 
 [Morishita 2012] \cite{M}.  
 Let $\mathfrak{M}^3$ be a category of  closed 3-dimensional manifolds,
such that  the arrows of $\mathfrak{M}^3$ are  homeomorphisms  between the  manifolds.
Likewise, let $\mathbf{K}$ be a category of the algebraic number fields,  where 
the arrows of $\mathbf{K}$ are  isomorphisms between such fields.
 Let  $\mathscr{M}^3\in \mathfrak{M}^3$ be a 3-manifold,  let $S^3\in \mathfrak{M}^3$ be the 3-sphere
 and let $O_K$ be the ring of integers of  $K\in\mathbf{K}$.  An exact relation between 3-manifolds and 
 number fields can be described  as follows. 
\begin{theorem}\label{thm2.1}
The exists a covariant functor $F: \mathfrak{M}^3\to \mathbf{K}$, such that:

\medskip
(i) $F(S^3)=\mathbf{Z}$; 

\smallskip
(ii)  each  ideal $I\subseteq O_K=F(\mathscr{M}^3)$ corresponds to 
a link $\mathscr{L}\subset\mathscr{M}^3$; 

\smallskip
(iii)  each  prime ideal $I\subseteq O_K=F(\mathscr{M}^3)$ corresponds to
a knot  $\mathscr{K}\subset\mathscr{M}^3$. 
\end{theorem}
Denote by  $\mathfrak{M}^4$ a category of  all smooth  4-dimensional manifolds $\mathscr{M}^4$,
such that  the arrows of $\mathfrak{M}^4$ are  homeomorphisms  between the  manifolds.
Denote by  $\mathfrak{K}$ a category of the hyper-algebraic number fields $\mathbb{K}$, such that
the arrows of $\mathfrak{K}$ are  isomorphisms between the fields.
Theorem \ref{thm2.1} extends to 4-manifolds as follows. 
\begin{theorem}\label{thm2.2}
{\bf (\cite[Theorem 1.1]{Nik1})}
The exists a covariant functor $F: \mathfrak{M}^4\to\mathfrak{K}$,
such that  the 4-manifolds $\mathscr{M}^4_1, \mathscr{M}^4_2\in \mathfrak{M}^4$ are homeomorphic
if and only if   the hyper-algebraic number fields
$F(\mathscr{M}^4_1), F(\mathscr{M}^4_2)\in  \mathfrak{K}$ are isomorphic. 
\end{theorem}

\subsection{Wedderburn Theorem}
Roughly speaking, Wedderburn's Theorem says that finite non-commutative fields 
cannot exist  [Maclagan-Wedderburn 1905] \cite{Wed1}.  Namely, denote by $\mathscr{D}$
a division ring.  Let   $\mathbb{F}_q$ be a finite field for some $q=p^r$, where $p$ is a prime
and $r\ge 1$ is an integer number.   
\begin{theorem}\label{thm2.3}
{\bf (Wedderburn Theorem)}
If $|\mathscr{D}|<\infty$ and $\mathscr{D}$ is finite dimensional over a division ring, 
then $\mathscr{D}\cong \mathbb{F}_q$ for some $q=p^r$.
\end{theorem}

\bigskip
We shall use  \ref{thm2.3} along with a classification of 
 simple rings due to Artin and Wedderburn.  Recall that a ring $R$
is called simple, if $R$ has only trivial two-sided ideals. 
By $M_n(\mathscr{D})$ we understand the ring of $n$ by $n$ matrices  over $\mathscr{D}$. 
\begin{theorem}\label{thm2.4}
{\bf (Artin-Wedderburn)}
If $R$ is a simple ring, then $R\cong M_n(\mathscr{D})$ for a division ring $\mathscr{D}$ 
and an integer $n\ge 1$. 
\end{theorem}
\begin{remark}\label{rmk2.5}
The ring $M_n(\mathscr{D})$ is a domain if and only if $n=1$. 
For instance, if $n=2$,  then the matrices 
$\left(\begin{smallmatrix} 1 & 0\cr 0 & 0 \end{smallmatrix}\right)$ and 
$\left(\begin{smallmatrix} 0 & 0\cr 0 & 1 \end{smallmatrix}\right)$
are  zero divisors in the ring $M_2(\mathscr{D})$. 
\end{remark}

\section{Proof of theorem \ref{thm1.1}}
Theorem \ref{thm1.1} will be proved by contradiction. Namely, we show that 
existence of  a non-trivial two-sided ideal in  $O_{\mathbb{K}}$ contradicts
 \ref{thm2.3}.   To begin, let us prove the following 
 lemma.  
\begin{lemma}\label{lem3.1}
$O_{\mathbb{K}}$ is a non-commutative Noetherian domain. 
\end{lemma}
\begin{proof}
Recall that $O_{\mathbb{K}}$ is generated by  the zeroes of a non-commutative polynomial
$\mathscr{P}(x):=\sum_i a_ixb_ixc_ix\dots e_ixl_i$,
where $a_i,b_i,c_i\dots, e_i, l_i\in  O_{\mathbb{L}}$ and $\mathbb{K}$ is a finite dimensional extension of  $\mathbb{L}$. 
By the Hilbert  Basis Theorem for non-commutative rings [Amitsur 1970] \cite{Ami1}, if $O_{\mathbb{L}}$ is Noetherian,  i.e. any  
ascending chain of the two-sided ideals of $O_{\mathbb{L}}$ 
stabilizes,  then the ring $O_{\mathbb{K}}$  is also  Noetherian. Repeating the construction, one
arrives at a finite dimensional  extension 
 $\mathbb{H}\subset\mathbb{K}$, where $\mathbb{H}$ is the field of quaternions.  
The ring of the Hurwitz quaternions  $O_{\mathbb{H}}$ is known to be Noetherian. 
Thus $O_{\mathbb{K}}$ is a Noetherian ring. Lemma \ref{lem3.1} is proved. 
\end{proof}

\bigskip
Returning to the proof of theorem \ref{thm1.1}, let us assume to the contrary, that $\mathbf{I}$
is a non-trivial two-sided ideal of $O_{\mathbb{K}}$. By lemma \ref{lem3.1},  there exists the maximal 
two-sided ideal $\mathbf{I}_{\max}$, such that 
\begin{equation}\label{eq3.2}
\mathbf{I}\subseteq\mathbf{I}_{\max}\subset O_{\mathbb{K}}. 
\end{equation}

\bigskip
\begin{lemma}\label{lem3.2}
 The ring $R:= O_{\mathbb{K}}/\mathbf{I}_{\max}$ is a simple domain.
\end{lemma}
\begin{proof}
The ring $R$ is simple, since $\mathbf{I}_{\max}$ is the maximal two-sided ideal of $O_{\mathbb{K}}$. 
The ring $R$ is a domain, since $O_{\mathbb{K}}$ is a domain and the homomorphism 
\begin{equation}\label{eq3.3}
h: O_{\mathbb{K}}\to R
\end{equation}
is surjective. 
\end{proof}
\begin{remark}\label{rmk3.3}
It follows  from $R\cong O_{\mathbb{K}}/\mathbf{I}_{\max}$,  that $|R|<\infty$.  Indeed, 
any non-trivial  subgroup of the abelian group $(O_{\mathbb{K}}, +)$ 
has finite index by the Margulis normal subgroup theorem.  In particular, the subgroup  $(\mathbf{I}_{\max}, +)$ 
has  finite index in   $(O_{\mathbb{K}}, +)$.  
\end{remark}

\bigskip
To finish the proof of theorem \ref{thm1.1},  we write 
\begin{equation}\label{eq3.4}
R\cong  M_n(\mathscr{D}),
\end{equation}
where $\mathscr{D}$ is a division ring,  see  Theorem \ref{thm2.4}. 
Since $R$ is a domain, we conclude that $n=1$ in formula (\ref{eq3.4}), see
remark \ref{rmk2.5}. Thus 
\begin{equation}\label{eq3.5}
R\cong  \mathscr{D}. 
\end{equation}

\bigskip
But remark \ref{rmk3.3} says that $|R|<\infty$  and by the Wedderburn Theorem one gets $R\cong  
 \mathbb{F}_q$ for some $q=p^r$. In particular, the homomorphism (\ref{eq3.3}) implies that 
the ring $O_{\mathbb{K}}$ is  commutative.
 Indeed, since $R$ is a commutative ring,  one gets $h(xy-yx)=h(x)h(y)-h(y)h(x)=h(x)h(y)-h(x)h(y)=0,$
 where $0$ is the neutral element of $R$ .
 In other words, the element $xy-yx$ belongs  to the kernel of $h$,  which is a two-sided ideal $I_h\subset
 O_{\mathbb{K}}$.  If $h$ is not injective, then $I_h$ is non-trivial  and taking  the multiplicative identity 
 $1\in I_h$ we obtain a contradiction  $h(1)=0$.  Thus $h$ is injective and $xy=yx$ for all $x,y\in O_{\mathbb{K}}$,
 i.e.  $O_{\mathbb{K}}$ is  a  commutative ring. 
  On the other hand,  the ring $O_{\mathbb{K}}$ cannot be commutative
 by an assumption of theorem \ref{thm1.1}. 
 The obtained contradiction completes the proof of theorem \ref{thm1.1}.

\bibliographystyle{amsplain}


\end{document}